\documentclass[12pt]{article}
\usepackage{amsmath}
\usepackage{amsthm}
\usepackage{amsfonts}
\usepackage{amssymb}

\newtheorem{thm}{Theorem}[section]
\newtheorem{lem}[thm]{Lemma}

\theoremstyle{definition}

\let\Lm=\Lambda
\let\ep=\epsilon
\let\al=\alpha
\let\vph=\varphi
\let\imp=\Rightarrow
\let\abs=\envert
\let\Abs=\wenvert
\let\sq=\sqrt
\let\wt=\widetilde
\newcommand{\floor}[1]{\left\lfloor#1\right\rfloor}
\newcommand{\Li}{\mathop{\mathrm{Li}}\nolimits}

\theoremstyle{remark}

\begin{document}
\title{Explicit formulae for primes in arithmetic progressions, II
}
\author{Tomohiro Yamada}
\date{}
\maketitle

\begin{abstract}
We shall give an explicit version of Bombieri-Vinogradov theorem
for moduli not divisible by an exceptional modulus.
\end{abstract}

\section{Introduction}\label{intro}

The well-known Bombieri-Vinogradov theorem, first proved by Bombieri\cite{Bom},
states that for any constant $A>0$, there exists some constant $B$ such that
\begin{equation}
\sum_{q<\frac{x^\frac{1}{2}}{\log^B x}}\max_{(a, q)=1}\abs{\pi(y; q, a)-\frac{\Li(y)}{\vph(q)}}=O\left(\frac{x}{\log^A x}\right),
\end{equation}
where $\pi(y; q, a)$ denotes the number of primes $\leq y$ congruent to $a\pmod{q}$.

Vaughan's result\cite{Vau} enables us to take $B=A+5$ (see Chapter 28 of Davenport's book \cite{DM}).
Timofeev showed that $B$ can be taken to be $A+\frac{3}{8}$.  Actually,
they have given upper bounds for the sum over large $q$'s and used
Siegel-Walfisz theorem in order to majorize the sum over small $q$'s.  As mentioned
by Harman\cite{Har}, their arguments are effective except the reference
to a possible exceptional zero.  Granville and Soundararajan, in their lecture note\cite{GS},
gave another effective method which enables us to take $B=A+3$ with the right
replaced by $O(x(\log\log x)^2\log^{-A} x)$.

Our purpose of this paper is to give an explicit formula of Bombieri-Vinogradov Theorem
using the method of Granville and Soundararajan.

Before stating our main result, we have to introduce an corollary of results of Kadiri\cite{Kad}
in order to refer a possible exceptional zero.

Define $\Pi(s, q)=\prod_{\chi\pmod{q}}L(s, \chi)$ and let $R_0=6.397$ and $R_1=2.0452\cdots$.
Theorem 1.1 of Kadiri\cite{Kad} states that the function $\Pi(s, q)$ has at most one zero
$\rho=\beta+it$ in the region $0\leq\beta<1-1/R_0\log \max\{q, q\abs{t}\}$,
which must be real and simple and induced by some nonprincipal real primitive character $\wt\chi\pmod{\wt{q}}$ with $987\leq \wt{q}\leq x$.
Moreover, Theorem 1.3 of \cite{Kad} implies that, for any given $Q_1$, such zero satisfies $\beta<1-1/2R_1\log Q_1$
except possibly one modulus below $Q_1$.

\begin{thm}\label{thm1}
Let $A$ be an integer with $2\leq A\leq 7$.
$C_0, x_0$ be the constants with $(C_0, \log\log x_0)=(C, Y_0)$ given in the column with $(\al_1, \al)=(A+3, A+1)$ of Table 4 in \cite[p.p.16--19]{Ymd}.
Moreover, let $q_0$ be the only (possible) modulus with $\log^\frac{3}{2}x<q_0\leq\log^{A+3} x$ such that
there exists a real zero $\beta\geq 1-1/2(A+3)R_1\log\log x$ for $\Pi(s, q_0)$.
If $x\geq x_0$, then
\begin{equation}
\sum_{q\leq\frac{x^\frac{1}{2}}{\log^{A+3} x}, q_0\nmid q}\max_{(a, q)=1}\abs{\psi(y; q, a)-\frac{y}{\vph(q)}}<\left(C_1+(A+3)(C_0+e^{-100})C_2^2\right)\frac{(\log\log x)^2}{2\log^A x}
\end{equation}
where
\begin{equation}
C_2=\prod_p\left(1+\frac{1}{p(p-1)}\right).
\end{equation}
and $C_1$ is the constant given in Table \ref{tbl2}.
\end{thm}

One of important applications of Bombieri-Vinogradov theorem is estimation of error terms in sieve formulae.
Sieve formulae often give error terms of the form
\begin{equation}
\sum_{d\mid P}\abs{\pi(x; q, a)-\frac{\pi(x)}{\vph(q)}}
\end{equation}
where $P$ denotes the product of primes below a given number.  When we plan to apply Bombieri-Vinogradov theorem
to such error terms, it suffices only to consider squarefree moduli.  This allows us to obtain a better upper bound as follows.
\begin{thm}\label{thm2}
Let $A, C_0, x_0$ be as in Theorem \ref{thm1}.
If $x\geq x_0$, then
\begin{equation}
\begin{split}
&\sum_{q\leq\frac{x^\frac{1}{2}}{\log^{A+3} x}, q_0\nmid q}\mu^2(q)\max_{(a, q)=1}\abs{\psi(y; q, a)-\frac{y}{\vph(q)}}\\
&\quad<\left(C_1^\prime+(A+3)(C_0+e^{-100})C_{13}\right)\frac{(\log\log x)^2}{2\log^A x},
\end{split}
\end{equation}
where
\begin{equation}
C_{13}=1+\frac{1.334}{(A+3)\log\log x_0}
\end{equation}
and $C_1^\prime$ is the constant given in Table \ref{tbl3}.
\end{thm}

We note that Granville and Stark \cite{GSt} showed that the abc conjecture,
which Mochizuki states that he proved\cite{Mch}, implies the nonexistence of Siegel zero
for characters with negative discriminants.  However, their method does not appear
to work for characters with positive discriminants.

For calculations of constants, we used PARI-GP.  Our script is available by requiring the author and can be used
to calculate constants except $C_1$ for arbitrary values of $A$.

\section{Notations and Preliminary lemmas}\label{lemmas}
Throughout this paper, we denote by $C_1, C_2, \ldots$ denote effectively computable constants
and $\theta$ denotes a quantity with absolute value $\leq 1$ not necessarily same
at each occurence.  Moreover, we denote the squarefree part of an integer $n$ by $n^*$.

For an arithmetic function $f(n)$ and the sum $F(x; q, a)=\sum_{n\leq x, n\equiv a\pmod{q}}f(n)$, we define
\begin{equation}
\begin{split}
F^{(R)}(x; q, a)= & F(x; q, a)-\frac{1}{\vph(q)}\sum_{\chi\pmod{q}}^{(\leq R)}\bar\chi(a)\sum_{n\leq x}f(n)\chi(n)\\
= & \frac{1}{\vph(q)}\sum_{\chi\pmod{q}}^{(>R)}\bar\chi(a)\sum_{n\leq x}f(n)\chi(n).
\end{split}
\end{equation}

For a sequence $(a_n)$, we put $\Abs{(a)}=\sq{\sum_n \abs{a_n}^2}$.
Moreover, $\sum_{\chi\pmod{(q)}}^{(r)}$ indicates the sum over all characters $\pmod{q}$ 
with primitive characters $\pmod{r}$ and
$\sum_{\chi\pmod{(q)}}^{(P)}$ indicates the sum over all characters $\pmod{q}$ 
each of which has its primitive character $\pmod{r}$ with $r$ satisfying the property $P$.

Now we shall introduce some preliminary lemmas, beginning by upper bounds for several quantities involving arithmetic functions.
\begin{lem}\label{lm0}
Let $Q_0$ be a constant greater than $223092870$.
For any $x, y>1, z>y$ and $q, u>Q_0$,
\begin{gather}
\prod_{y\leq p<z}\frac{p}{p-1}< 2\frac{\log z}{\log y},\\
\frac{q}{\vph(q)}<C_3\log\log q,\\
\sum_{n\leq x}\frac{1}{\vph(n)}<C_2(1+\log x),\\
\sum_{n\leq u}\frac{\mu^2(n)n}{\vph^2(n)}\leq C_4\log u,\\
\sum_{n\leq x}\frac{\mu^2(n)}{\vph^2(n)}<C_5
\end{gather}
and
\begin{equation}
\pi(x)<\frac{2x}{\log x},
\end{equation}
where 
\begin{gather}
C_3=e^\gamma+\frac{5}{2\log\log Q_0},\\
C_4=C_2+\frac{\frac{89}{16}-C_2\log 6}{\log Q_0},\\
\end{gather}
and
\begin{equation}
C_5=\prod_p\left(1+\frac{1}{(p-1)^2}\right).
\end{equation}

\end{lem}
\begin{proof}
It follows from \cite[Theorem 7--8, p.70]{RS} that
\begin{equation}
\prod_{y\leq p<z}\frac{p}{p-1}<C_1\frac{\log z}{\log y}.
\end{equation}
By \cite[Theorem 15, (3.41), p.p.71--72]{RS}, we have $\vph(n)>C_3 n/\log\log n$.
We can see that $\sum_{m<x}\frac{1}{\vph(m)}<C_2(1+\log x)$ from the argument
in the proof of Theorem A. 17 in \cite[p. 316]{Nat}.
Moreover, it is easy to see that $\sum_{n\leq x}\frac{\mu^2(n)}{\vph^2(n)}<\sum_{n}\frac{\mu^2(n)}{\vph^2(n)}=C_5$.

In order to show $\sum_{n\leq x}\frac{\mu^2(n)n}{\vph^2(n)}<C_2\log x+C_4$,
we begin by estimating $Q_m(x_1, x)=\sum_{x_1<n\leq x, (n, m)=1}\mu^2(n)/n$,
the sum of reciprocals of squarefree integers $n$ with $x_1<n\leq x$.

Let $Q(x)$ the number of squarefree integers $\leq x$.  Then we have
\begin{equation}
Q(x)=\sum_m\mu(m)\floor{\frac{x}{m^2}}=x\sum_{m\leq\sq{x}}\frac{\mu(m)}{m^2}+\theta Q(\sq{x})
=\frac{x}{\zeta(2)}+2\theta\sq{x}
\end{equation}
and therefore
\begin{equation}\label{eq21}
\begin{split}
Q_m(x_1, x)=&\sum_{l\mid m}\mu(l)\sum_{\substack{k\leq\sq{x},\\ (k, m)=1}}\mu(k)\sum_{\substack{k^2l\mid n,\\ x_1<n\leq x}}\frac{1}{n}
=\sum_{l\mid m}\frac{\mu(l)}{l}\sum_{\substack{k\leq\sq{x},\\ (k, m)=1}}\frac{\mu(k)}{k^2}\sum_{\frac{x_1}{k^2 l}<n\leq\frac{x}{k^2 l}}\frac{1}{n}\\
=&\sum_{l\mid m}\frac{\mu(l)}{l}\sum_{\substack{k\leq\sq{x},\\ (k, m)=1}}\frac{\mu(k)}{k^2}\left(\log\frac{x}{x_1}+\theta k^2l\left(\frac{1}{x_1}+\frac{1}{x}\right)\right)\\
=&\log\frac{x}{x_1}\left(\sum_{l\mid m}\frac{\mu(l)}{l}\right)\left(\sum_{\substack{k\leq\sq{x},\\ (k, m)=1}}\frac{\mu(k)}{k^2}\right)
+\theta d(m^*)\sq{x}\left(\frac{1}{x_1}+\frac{1}{x}\right).
\end{split}
\end{equation}

Since
\begin{equation}
\begin{split}
\sum_{\substack{x_1<q\leq x,\\ (q, l)=1}}\frac{\mu^2(q)}{\vph(q)}=&\sum_{\substack{x_1<q\leq x,\\ (q, l)=1}}\frac{\mu^2(q)}{q}\prod_{p\mid q}\frac{p}{p-1}
=\sum_{\substack{x_1<q\leq x,\\ (q, l)=1}}\frac{\mu^2(q)}{q}\prod_{m^*\mid q}\frac{1}{m}\\
=&\sum_{m=1}^\infty\frac{1}{m}\sum_{\substack{m^*\mid q,\\ x_1<q\leq x,\\ (q, l)=1}}\frac{\mu^2(q)}{q}=\sum_{m=1}^\infty\frac{1}{mm^*}\sum_{\substack{x_1/m^*<r\leq x/m^*,\\ (r, ml)=1}}\frac{\mu^2(r)}{r}\\
<&\sum_{m=1}^\infty\frac{1}{mm^*}Q_{ml}\left(\frac{x_1}{m^*},\frac{x}{m^*}\right),
\end{split}
\end{equation}
(\ref{eq21}) gives
\begin{equation}\label{eq22}
\begin{split}
\sum_{\substack{x_1<q\leq x,\\ (q, l)=1}}\frac{\mu^2(q)}{\vph(q)}=&\sum_{m\leq x}\frac{1}{mm^*}\left[\left(\frac{h_2(ml)}{\zeta(2)}+\theta\frac{m^*}{x}\right)\frac{\vph(ml)}{ml}\log\frac{x}{x_1}+\theta d((ml)^*)\sq{m}\left(\frac{\sq{x}}{x_1}+\frac{1}{\sq{x}}\right)\right]\\
=&\frac{\vph(l)}{l}\log\frac{x}{x_1}+\theta\left[B_1(l)\frac{\log\frac{x}{x_1}}{\sq{x}}+B_2(l)\left(\frac{\sq{x}}{x_1}+\frac{1}{\sq{x}}\right)\right],
\end{split}
\end{equation}
where $B_1(l), B_2(l)$ and $h_2(n)$ are arithmetic functions defined by
\begin{equation}
B_1(l)=\frac{\zeta(3/2)}{\zeta(3)}\prod_{p\mid l}\frac{\sq{p}(p-1)}{p^\frac{3}{2}+1},
\end{equation}
\begin{equation}
B_2(l)=2^{\omega(l)}\prod_p\left(1+\frac{2}{\sq{p}(p-1)}\right)\prod_{p\mid l}2\left(1+\frac{2}{\sq{p}(p-1)}\right)^{-1}
\end{equation}
and
\begin{equation}
h_2(n)=\prod_{p\mid n}\left(1-\frac{1}{p^2}\right)^{-1}.
\end{equation}

\begin{equation}
\begin{split}
\sum_{x_1<n\leq x}\frac{\mu^2(n)n}{\vph^2(n)}=&\sum_{x_1<n\leq x}\frac{\mu^2(n)}{n}\sum_{l\mid n}\frac{\mu^2(l)}{l}=\sum_l\frac{\mu^2(l)}{\vph(l)}\sum_{\substack{x_1<q\leq x,\\ l\mid q}}\frac{\mu^2(q)}{\vph(q)}\\
=&\sum_l\frac{\mu^2(l)}{\vph^2(l)}\left(\frac{\vph(l)}{l}\log\frac{x}{x_1}+\theta\left[B_1(l)\frac{\log\frac{x}{x_1}}{\sq{x}}+B_2(l)\left(\frac{\sq{x}}{x_1}+\frac{1}{\sq{x}}\right)\right]\right)\\
=&C_2\log\frac{x}{x_1}+\theta\left[B_3\frac{\log\frac{x}{x_1}}{\sq{x}}+B_4\left(\frac{\sq{x}}{x_1}+\frac{1}{\sq{x}}\right)\right],
\end{split}
\end{equation}
where
\begin{equation}
B_3=\sum_l\frac{\mu^2(l)}{\vph^2(l)}B_1(l)=\frac{\zeta(3/2)}{\zeta(3)}\prod_{p}\left(1+\frac{p}{(p-1)(p^\frac{3}{2}+1)}\right)
\end{equation}
and
\begin{equation}
B_4=\sum_l\frac{\mu^2(l)}{\vph^2(l)}B_2(l)=\prod_{p}\left(1+\frac{2}{\sq{p}(p-1)}\right)\left(1+\frac{2\sq{p}}{(p-1)^2\left(1+\frac{2}{\sq{p}(p-1)}\right)}\right).
\end{equation}

Now we can confirm $\sum_{n\leq x}\frac{\mu^2(n)n}{\vph(n)}\leq C_2\log x+\frac{89}{16}-C_2\log 6\leq C_4\log x$ for all $x>1$ by calculation.

Finally, using the formula of Rosser and Schoenfeld \cite[p. 69, Theorem 1. (3.2)]{RS},
we have $\pi(x)<\frac{2x}{\log x}$ for any $x>1$.  An elementary proof is implicit in \cite[Chap. 22]{HW}.
\end{proof}

Next, we shall introduce explicit versions of Propositions in Chapter 13 of \cite{GS}.
\begin{lem}\label{lm1}
\begin{equation}
\begin{split}
&\sum_{q\leq Q}\sum_{\chi\pmod{q}}^{(\geq R)} \abs{\sum_{m=X+1}^{X+M}a_m\chi(m)}\abs{\sum_{n=Y+1}^{Y+N}b_n\chi(n)}\\
&\quad <C_3\log\log Q\Abs{(a)}\Abs{(b)}\\
&\qquad \times \left(\frac{\al_1}{\al_1-1}C_4\left(\frac{\sq{MN}}{R}\log Q+\frac{\al_1}{\log\al_1}(\sq{M}+\sq{N})\log^2 Q\right)+\frac{\al_1^3}{\al_1-1}C_5 Q\right).
\end{split}
\end{equation}
\end{lem}

\begin{proof}
Proceeding similarly to the proof of Proposition 13.2 in \cite[p.p.73--74]{GS}, we obtain
\begin{equation}
\begin{split}
&\sum_{q\leq Q}\sum_{\chi\pmod{q}}^{(\geq R)} \abs{\sum_{m=X+1}^{X+M}a_m\chi(m)}\abs{\sum_{n=Y+1}^{Y+N}b_n\chi(n)}\\
&\quad =\sum_{l\leq Q}\frac{\mu^2(l)}{\vph(l)}\sum_{R\leq r\leq\frac{Q}{l}, (r, l)=1}\frac{1}{\vph(r)}\sum_{\substack{s\leq Q/rl,\\p\mid s\imp p\mid rl}}\frac{1}{s}\\
&\qquad \times \sum_{\chi\pmod{r}}^* \abs{\sum_{\substack{X+1\leq m\leq X+M,\\ (m, l)=1}}a_m\chi(m)}\abs{\sum_{\substack{Y+1\leq n\leq Y+N,\\ (n, l)=1}}b_n\chi(n)}\\
&\quad<C_3\log\log Q
\sum_{l\leq Q}\frac{\mu^2(l)l}{\vph(l)}\sum_{\substack{y=\frac{Q}{l\al_1^i}, 0\leq i\leq I}}\frac{1}{y}\sum_{\al_1^{-1}y\leq r<y, (r, l)=1}\frac{r}{\vph(r)}\\
&\qquad \times \sum_{\chi\pmod{r}}^* \abs{\sum_{\substack{X+1\leq m\leq X+M,\\ (m, l)=1}}a_m\chi(m)}\abs{\sum_{\substack{Y+1\leq n\leq Y+N,\\ (n, l)=1}}b_n\chi(n)},
\end{split}
\end{equation}
where $I=\floor{\log(Q/lR)/\log\al_1}$.

Using Cauchy's theorem and large-sieve inequality, this is
\begin{equation}
\begin{split}
< & C_3\log\log Q \Abs{(a)}\Abs{(b)} \\
& \quad \times \sum_{l\leq Q}\frac{\mu^2(l)l}{\vph(l)}\sum_{\substack{y=\frac{Q}{l\al_1^i}, 0\leq i\leq I}}\frac{1}{y}
\left(\frac{\sq{MN}}{y}+(\sq{M}+\sq{N})+y\right)\\
\leq & C_3\log\log Q\Abs{(a)}\Abs{(b)}\\
& \quad \times \sum_{l\leq Q}\frac{\mu^2(l)l}{\vph(l)}
\left(\frac{\al_1}{\al_1-1}\left(\frac{\al_1\sq{MN}}{R}+\frac{Q}{l}\right)+(\sq{M}+\sq{N})\frac{\log\frac{Q}{lR}}{\log\al_1}\right)\\
\leq & C_3\log\log Q\Abs{(a)}\Abs{(b)}\\
& \quad \times \left(\frac{\al_1^2}{\al_1-1}C_4\left(\frac{\sq{MN}}{R}\log Q+\frac{1}{\log\al_1}(\sq{M}+\sq{N})\log^2 Q\right)+\frac{\al_1}{\al_1-1}C_5 Q\right).
\end{split}
\end{equation}
\end{proof}

\begin{lem}\label{lm2}
\begin{equation}
\sum_{q\leq Q}\frac{1}{\vph(q)}\sum_{\chi\pmod{q}}^{(\geq R)} \abs{\sum_{u\leq x}c_u\chi(u)}\leq AB(C_6 xR^{-1}+C_7 Q\sq{x})\log^2 x\log\log Q.
\end{equation}
\end{lem}
\begin{proof}
We see that $\sum_{u\leq x}c_u\chi(u)=\sum_{mn\leq x}a_m\chi(m)b_n\chi(n)$.
As in the proof of Proposition 13.6 in \cite[p.p. 76--77]{GS}, we use the partition for $m$ in the range $X<m\leq 2X$.

Let $Y=x/X$ and $I_{j, k}$ be the interval
\begin{equation}
\left(1+\frac{2j}{2^k}\right)X<m\leq \min\left\{1+\frac{2j+1}{2^k}, 2\right\}X, \frac{Y}{1+\frac{2j+2}{2^k}}<n\leq\frac{Y}{1+\frac{2j+1}{2^k}}.
\end{equation}
Let $S_1(X)$ be the sum over the $I_{j, k}$'s with $0\leq j\leq 2^{k-1}-1, k=1, 2, \ldots, K$
and $S_2(X)$ be the sum over
\begin{equation*}
\left(1+\frac{j}{2^K}\right)X<m\leq \min\left\{1+\frac{j+1}{2^K}, 2\right\}X, \frac{Y}{1+\frac{j+1}{2^K}}<n\leq\frac{x}{m}
\end{equation*}
with $0\leq j\leq 2^K-1$.

By Lemma \ref{lm1}, the sum over the interval $X<m\leq X+M, Y<n\leq Y+N$ is at most
\begin{equation}
\begin{split}
&C_3\log\log Q AB\sq{MN}\\
&\quad\times\left(\frac{\al_1}{\al_1-1}C_4\frac{\sq{MN}}{R}\log Q+\al_1(\al_1-1)C_5 Q+\frac{\al_1-1}{\log\al_1}C_4(\sq{M}+\sq{N})\log^2 Q\right).
\end{split}
\end{equation}
For each interval $I_{j, k}$, we have $M\leq 2^{-k}X+1$ and $N\leq 2^{-k}Y+1$.
Running over the $I_{j, k}$'s with $0\leq j\leq 2^{k-1}-1, k=1, 2, \ldots, K$, we have
\begin{equation}
\sum\sq{MN}<\frac{1}{\log 2}\sq{XY}\log X+\frac{1}{2-\sq{2}}(X+\sq{XY})+X,
\end{equation}
\begin{equation}
\sum M\sq{N}<\frac{1}{2(\sq{2}-1)}X\sq{Y}+\frac{1}{2-\sq{2}}\left(\sq{XY}+\frac{X\sq{X}}{2\sq{Y}}\right)+\frac{1}{4-\sq{2}}\frac{X\sq{X}}{Y},
\end{equation}
\begin{equation}
\sum\sq{M}N<\frac{1}{2(\sq{2}-1)}\sq{X}Y+\frac{1}{2-\sq{2}}\left(X+\frac{Y}{2}\right)+\frac{1}{4-\sq{2}}X
\end{equation}
and
\begin{equation}
\sum MN<\frac{XY}{2}+\frac{1}{2\log 2}(X+Y)\log X+X.
\end{equation}

Taking the total upper bound, we have
\begin{equation}
\begin{split}
& S_1(X)\leq C_3 AB\log\log Q\\
&\quad <\biggl[
 \frac{\al_1^2C_4}{(\al_1-1)R}\log Q\left(
  \frac{x}{2}+\frac{1}{2\log 2}\left(X+\frac{x}{X}\right)\log X+X
 \right)\\
&\qquad +\frac{\al_1}{\al_1-1}C_5 Q\left(
  \frac{\sq{x}\log X}{4\log^2 2}+\frac{1}{2-\sq{2}}(X+\sq{x})+X
 \right)\\
&\qquad +\frac{C_4}{\log\al_1}\log^2 Q\\
&\qquad\quad\times\left\{
  \frac{\sq{2}(\sq{2}+1)}{2(\sq{2}-1)}\left(\sq{xX}+\frac{x}{\sq{X}}\right)\right.\\
&\qquad\qquad\left.+\frac{1}{2-\sq{2}}\left(X+\frac{x}{2X}+\sq{x}+\frac{X^2}{2\sq{x}}\right)+\frac{1}{4-\sq{2}}\left(X+\frac{X^2}{2\sq{x}}\right)
 \right\}
\biggr]
\end{split}
\end{equation}
for each $X$.  Summing this over $X=2^i R^2$ with $0\leq i\leq \log(\sq{x}/R^2)/\log 2$, we have 
\begin{equation}
\begin{split}
&\sum_X S_1(X)\leq C_3 AB\log\log Q\\
&\quad <\biggl[
 \frac{\al_1^2C_4}{(\al_1-1)R}\log Q\left(
  \frac{x\log x}{4\log 2}+\frac{1}{2\log 2}\left(\frac{2\log 2}{z_1(\al_1)^2}x+\frac{\sq{x}\log^2 x}{4\log 2}\right)+2\sq{x}
 \right)\\
&\qquad +\frac{\al_1}{\al_1-1}C_5 Q\left\{
  \frac{\sq{x}\log x}{4\log^2 2}\left(\log z_1+\frac{1}{4}\log x\right)\right.\\
&\left.\qquad\qquad+\left((3-\sq{2})(2+\sq{2})+\frac{1}{2-\sq{2}}\frac{\log x}{2\log 2}\right)\sq{x}
 \right\}\\
&\qquad +\frac{C_4}{\log\al_1}\log^2 Q\left\{
  \frac{\sq{2}(\sq{2}+1)}{2(\sq{2}-1)}\left(x^\frac{3}{4}+\frac{x}{z_1}\right)\right.\\
&\left.\qquad\qquad  +(2+\sq{2})\left(\frac{4\sq{x}}{3}+\frac{x}{2z_1}+\frac{\sq{x}\log x}{4\log 2}\right)+\frac{8}{3(4-\sq{2})}\sq{x}
 \right\}
\biggr].
\end{split}
\end{equation}

For the sum $S_2(X)$, we see that $x-2<(1+2^{-K})^{-1}x<mn\leq (1+2^{-K})x\leq x+2$.  Thus we have
\begin{equation}
S_2(X)\leq \sum_{q\leq Q}\frac{1}{\vph(q)}\sum_{\chi\pmod{q}}^{(\geq R)}AB\sum_{x-2<l\leq x+2}d(l)<ABQ\sum_{x-2<l\leq x+2}d(l).
\end{equation}

By \cite{NR}, we know that $d(l)<l^{1.06602/\log\log l}$ and therefore
\begin{equation}
\sum_X S_2(X)\leq 4ABQ\frac{\log x}{\log 2} (x+2)^\frac{1.6602}{\log\log (x+2)}<e^{-1000}ABQx^\frac{1}{2}\log^2 x\log\log Q.
\end{equation}
This completes the proof.
\end{proof}

\section{Proof of Theorem 1.1}\label{proof}
Let $g(n)$ be the totally multiplicative function defined by $g(p)=0$ if $p\leq R^2$ and $g(p)=1$ if $p>R^2$.
We define $a_n$ by $a_n=g(n)\mu(n)$ if $n>1$ and $a_1=0$ and $b_m=g(m)\log m$.
Let $G(x, \chi)=\sum_{n\leq x}g(n)\chi(n)\log n, G(x; q, a)=\sum_{n\leq x, n\equiv a\pmod{q}}g(n)\log n$.

Now we shall proceed as in p.132 of \cite{GS2}.
We observe that $c_n=\Lm_{R^2}(n)-g(n)\log n$ and therefore we have, using Lemma \ref{lm2},
\begin{equation}\label{eq31}
\sum_{q\leq Q}\frac{1}{\vph(q)}\sum_{\chi}^{(\geq R)}\abs{\sum_{u\leq x}c_u\chi(u)}\leq (C_6 xR^{-1}+C_7Q\sq{x})\log^3 x\log\log Q.
\end{equation}

Using a trivial estimate
\begin{equation}
\begin{split}
\abs{\sum_{n\leq x}\Lm(n)\chi(n)-\Lm_{R^2}(n)\chi(n)}\leq&\sum_{n\leq x}\abs{\Lm(n)-\Lm_{R^2}(n)}\leq\pi(R^2)\log x\\
<& 2R^2\frac{\log x}{\log R}<C_8\sq{x}\log^3 x\log\log Q,
\end{split}
\end{equation}
we have
\begin{equation}\label{eq32}
\sum_{q\leq Q}\frac{1}{\vph(q)}\sum_{\chi\pmod{q}}^{(\geq R)}\abs{\sum_{n\leq x}\Lm(n)\chi(n)-\Lm_{R^2}(n)\chi(n)}< C_8Q\sq{x}\log^3 x\log\log Q.
\end{equation}

Using (\ref{eq32}), we can replace $c_u$ in (\ref{eq31}) by $\Lm(u)-g(u)\log u$ to obtain
\begin{equation}
\begin{split}
&\sum_{q\leq Q}\frac{1}{\vph(q)}\sum_{\chi\pmod{q}}^{(\geq R)}\abs{\sum_{n\leq x}\Lm(n)\chi(n)-g(n)\chi(n)\log n}\\
&\quad<(C_6 xR^{-1}+(C_7+C_8)Q\sq{x})\log^3 x\log\log Q
\end{split}
\end{equation}
and therefore
\begin{equation}\label{eq33}
\begin{split}
\sum_{q\leq Q}\abs{\psi^{(R)}(x; q, a)-G^{(R)}(x; q, a)}\leq & \sum_{q\leq Q}\frac{1}{\vph(q)}\sum_{\chi\pmod{q}}^{(\geq R)}\abs{\psi(x, \chi)-G(x, \chi)}\\ < & (C_6 xR^{-1}+(C_7+C_8)Q\sq{x})\log^3 x\log\log Q.
\end{split}
\end{equation}

We observe that
\begin{equation}
\begin{split}
G^{(R)}(x; q, a)= & G(x; q, a)-\frac{1}{\vph(q)}\sum_{r\leq R, r\mid q}\sum_{\chi\pmod{q}}^{r}\bar\chi(a)\sum_{\substack{b\pmod{q},\\ (b, q)=1}}\chi(b)G(x; q, b)\\
= & \frac{G(x)}{\vph(q)}+G^{(1)}(x; q, a)\\
 & -\frac{1}{\vph(q)}\sum_{r\leq R, r\mid q}\sum_{\chi\pmod{q}}^{r}\bar\chi(a)\sum_{\substack{b\pmod{q},\\ (b, q)=1}}\chi(b)\left(\frac{G(x)}{\vph(q)}+G^{(1)}(x; q, b)\right)\\
= & G^{(1)}(x; q, a)-\frac{1}{\vph(q)}\sum_{r\leq R, r\mid q}\sum_{\chi\pmod{q}}^{r}\bar\chi(a)\sum_{\substack{b\pmod{q},\\ (b, q)=1}}\chi(b)G^{(1)}(x; q, b)
\end{split}
\end{equation}
to obtain
\begin{equation}\label{eq34}
\begin{split}
\sum_{q\leq Q}\abs{G^{(R)}(x; q, a)}\leq& \sum_{q\leq Q}\abs{G^{(1)}(x; q, a)}+\sum_{q\leq Q}\frac{1}{\vph(q)}\sum_{r\leq R, r\mid q}\sum_{\chi\pmod{q}}^{(r)}\sum_{\substack{b\pmod{q},\\ (b, q)=1}}\abs{G^{(1)}(x; q, b)}\\
=& \sum_{q\leq Q}\abs{G^{(1)}(x; q, a)}+\sum_{r\leq R}\sum_{q\leq Q, r\mid q}\frac{1}{\vph(q)}\sum_{\substack{b\pmod{q},\\ (b, q)=1}}\abs{G^{(1)}(x; q, b)}.
\end{split}
\end{equation}

Now we shall bound $G^{(1)}(x; q, b)$ for each congruent class $b\pmod{q}$.

\begin{lem}\label{lm3}
Let $c=c(A)$ be a constant depending only on $A$.
If $x\geq \exp cA^2\log A$, then
\begin{equation}
G^{(1)}(x; q, b)<\left(1+\frac{1}{\log x}\right)\frac{x\log\log x}{e^\gamma \vph(q)\log^{A+2} x}+\left(1+\frac{1}{\vph(q)}\right)\frac{x}{Q\log^{A+1}x}.
\end{equation}

We can take $c$ as in table \ref{tbl1} for $A=2, 3, \ldots, 7$ and $c=17$ for $A>e^{160.51440939}$.

\end{lem}
\begin{table}
\caption{constants in Lemma \ref{lm3}}\label{tbl1}
\begin{center}
\begin{small}
\begin{tabular}{| c | c |}
 \hline
$A$ & $cA^2\log A$ \\
 \hline
$2$ & $6978$ \\
$3$ & $9805$ \\
$4$ & $13116$ \\
$5$ & $16912$ \\
$6$ & $21193$ \\
$7$ & $25962$ \\
 \hline
\end{tabular}
\end{small}
\end{center}
\end{table}

\begin{proof}
Let $V(z)=\prod_{p\leq z}(1-1/p)$ and $N(y, z; q, a)$ be the number of integers $n\leq y, n\equiv a\pmod{q}$ such that 
$n$ has no prime factor $\leq z$.

We see from \cite[Theorem 8, p. 70]{RS} that for any $z>w>1$, we have
\begin{equation}
\prod_{w\leq p<z}\frac{p}{p-1}<2\frac{\log z}{\log w}.
\end{equation}

Let $D=\frac{x}{Q\log^{A+2} x}=x^\frac{1}{2}\log x$,
\begin{equation}
s=\frac{\log D}{2\log R}=\frac{\log x}{4(A+3)\log\log x}+\frac{1}{2A+6}
\end{equation}
and $b=4+s/\log(2s)$.

Now, confirming that $1-e^{1+1/b}/b>1/e$ and $s>b+1$ for our choices of $c$ and $A$,
\cite[Theorem 3.3.1, p. 91]{Gre} gives
\begin{equation}
\begin{split}
&\abs{N(y, R^2; q, b)-\frac{y}{\vph(q, R^2)}V(R^2)}\leq \\
&\qquad\frac{y}{\vph(q, R^2)}V(R^2)\left[2\left(4+\frac{s}{\log 2s}\right)\right]^5 Ks e^{s(-\log s+3+\log\log K)}+D,
\end{split}
\end{equation}
where $\vph(q, R^2)=q\prod_{p\mid q, p>R^2}(1-p^{-1})$.  Clearly $\vph(q, R^2)\geq\vph(q)$.

Since
\begin{equation}
G(x; q, b)=(\log x)N(x, R^2; q, b)-\int_1^x \frac{N(t, R^2; q, b)}{t}dt
\end{equation}
and
\begin{equation}
G^{(1)}(x; q, b)=G(x; q, b)-\frac{1}{\vph(q)}G(x),
\end{equation}
we have
\begin{equation}
\begin{split}
&\abs{G^{(1)}(x; q, b)}\\
\leq &(\log x)\abs{N(x, R^2; q, b)-\frac{1}{\vph(q)}N(x, R^2)}
+\int_1^x\abs{\frac{N(t, R^2; q, b)}{t}-\frac{N(t, R^2)}{t\vph(q)}}dt\\
\leq &(1+\log x)\frac{2x}{\vph(q)}V(R^2)\left[2\left(4+\frac{s}{\log 2s}\right)\right]^5 2s e^{s(-\log s+3+\log\log 2)}\\
&+D\left(1+\frac{1}{\vph(q)}\right)\log x.
\end{split}
\end{equation}

Let
\begin{equation}
\ep_1=1-\frac{4(A+3)s\log s}{\log x}
\end{equation}
and
\begin{equation}
\ep_2=\frac{s(3+\log\log 2)+\log 2s+5\log\left[2\left(4+\frac{s}{\log 2s}\right)\right]}{s\log s}.
\end{equation}
Then we can easily see that $\ep_1, \ep_2$ tend to zero as $c$ tends to infinity.

Now we take $c$ so that
\begin{equation}
\begin{split}
&\frac{c(1-\ep_1)(1-\ep_2)A^2\log A}{4(A+3)}+2\log\log R-\log(1+8\log^{-2} R)\\
&\qquad\qquad>(A+3)\log(cA^2\log A).
\end{split}
\end{equation}

Since 
\begin{equation}
s\log s=\frac{\log x}{4(A+3)}(1-\ep_1)>\frac{c(1-\ep_1)A^2\log A}{4(A+3)},
\end{equation}
\begin{equation}
e^{(1-\ep_2)s\log s}=\left[2\left(4+\frac{s}{\log 2s}\right)\right]^5 2s e^{s(-\log s+3+\log\log 2)}
\end{equation}
and $V(R^2)<e^{-\gamma}(1+1/8\log^2 R)/2\log R$ by \cite[Theorem 7, (3.26), p. 70]{RS},
we have
\begin{equation}
\begin{split}
&G^{(1)}(x; q, b)\\
<&\left(1+\frac{1}{\log x}\right)\frac{x\log\log x}{e^\gamma \vph(q, R^2)\log^{A+2} x}+D\left(1+\frac{1}{\vph(q, R^2)}\right)\log x\\
\leq& \left(1+\frac{1}{\log x}\right)\frac{x\log\log x}{e^\gamma \vph(q)\log^{A+2} x}+\left(1+\frac{1}{\vph(q)}\right)\frac{x}{Q\log^{A+1}x},
\end{split}
\end{equation}
as stated in the Lemma.
\end{proof}

Now the assumption $x\geq x_0$ allows us to apply Lemma \ref{lm3} since Table 4 in \cite{Ymd} implies $x_0>\exp (cA^2\log A)$ for any $A$.
For the first sum in (\ref{eq34}), seeing that $\sum_{q\leq Q}\frac{1}{\vph(q)}<C_2(1+\log Q)<\frac{C_2}{2}(\log x-1)$, we have
\begin{equation}\label{eq35}
\sum_{q\leq Q}\abs{G^{(1)}(x; q, a)}< \frac{C_2x\log\log x}{2e^\gamma\log^{A+1} x}+\left(1+\frac{C_2\log x}{2Q}\right)\frac{x}{\log^{A+1} x}.
\end{equation}
For the second sum in (\ref{eq34}), we apply Lemma \ref{lm3} to $G^{(1)}(x; q, b)$ for each congruent class $b\pmod{q}$ to obtain
\begin{equation}\label{eq36}
\begin{split}
& \sum_{r<R}\sum_{q\leq Q, r\mid q}\frac{1}{\vph(q)}\sum_{\substack{b\pmod{q},\\ (b, q)=1}}\abs{G^{(1)}(x; q, b)} \\
&\quad \leq \sum_{r<R}\sum_{q\leq Q, r\mid q}\max_{b\pmod{q}, (b, q)=1}\abs{G^{(1)}(x; q, b)}\\
&\quad <\sum_{r<R}\sum_{q\leq Q, r\mid q}\left(1+\frac{1}{\log x}\right)\frac{x\log\log x}{e^\gamma \vph(q)\log^{A+2} x}+\left(1+\frac{1}{\vph(q)}\right)\frac{x}{Q\log^{A+1}x}\\
&\quad <\sum_{r<R}\frac{C_2x\log\log x}{2\vph(r)e^\gamma \log^{A+1} x}+\left(\frac{1}{r}+\frac{C_2\log x}{2Q\vph(r)}\right)\frac{x}{\log^{A+1}x}\\
&\quad <\frac{C_2^2(1+(A+3)\log\log x)x\log\log x}{2e^\gamma \log^A x}+\\
&\quad \left(1+(A+3)\log\log x+\frac{C_2^2(1+(A+3)\log\log x)\log x}{2Q}\right)\frac{x}{\log^{A+1}x}.
\end{split}
\end{equation}

With the aid of (\ref{eq35}) and (\ref{eq36}), (\ref{eq33}) gives
\begin{equation}
\sum_{q\leq Q}\abs{\psi^{(R)}(x; q, a)}\leq(C_6 xR^{-1}+(C_7+C_8)Q\sq{x})\log^3 x\log\log Q+C_9\frac{x(\log\log x)^2}{\log^A x}.
\end{equation}
Since $R=\log^{A+3} x$ and $Q=\sq{x}/\log^{A+3} x$, we have
\begin{equation}\label{eq37}
\sum_{q\leq Q}\abs{\psi^{(R)}(x; q, a)}<\left(\frac{C_6+C_7+C_8}{\log\log x}+C_9\right)\frac{x(\log\log x)^2}{\log^A x}.
\end{equation}

Now it suffices to majorize
\begin{equation}\label{eq38}
\sum_{q\leq Q, q_0\nmid q}\abs{\psi^{(R)}(x; q, a)-\psi^{(1)}(x; q, a)}=\sum_{q\leq Q, q_0\nmid q}\abs{\frac{1}{\vph(q)}\sum_{\substack{1<r\leq R,\\ r\mid q}}\sum^{(r)}\bar\chi(a)\psi(x, \chi)}.
\end{equation}

The proof of Theorem 3.6 of \cite{MC} shows that the left-hand side quantity $\frac{\vph(k)}{x}\abs{\psi(x; k, l)-\frac{x}{\vph(k)}}$
in this theorem can be replaced by $-1+x^{-1}\sum_{\chi\pmod{k}}\abs{\psi(x, \chi)}$.  This also applies to Theorem 1.1 of \cite{Ymd}.

Now, proveded that $x\geq x_0$, Theorem 1.1 of \cite{Ymd} gives
\begin{equation}\label{eq39}
\sideset{}{^*}\sum_{\chi\pmod{q}}\abs{\psi(x, \chi^*)}<\frac{C_0x\log\log x}{\log^{A+1} x}+E_0\left(\frac{x^{-\beta_0}}{1-\beta_0}+\frac{x^{\beta_0-1}}{\beta_0}\right),
\end{equation}
where $E_0=1$ and $\beta_0$ denote the Siegel zero modulo $q$ if it exists and $E_0=0$ otherwise.

If $q\leq R$ is non-exceptional, then the right-hand side of (\ref{eq39}) is at most
\begin{equation}
\frac{C_0\log\log x}{\log^{A+1} x}+\frac{x^{-\frac{1}{2(A+3)R_1\log\log x}}}{1-\frac{1}{2(A+3)R_1\log\log x}}+2x^\frac{1}{2}.
\end{equation}
We can easily confirm that, provided that $x\geq x_0$,
\begin{equation}
\frac{x^{-\frac{1}{2(A+3)R_1\log\log x}}}{1-\frac{1}{2(A+3)R_1\log\log x}}+2x^\frac{1}{2}<e^{-100}\frac{x\log\log x}{\log^{A+1} x}.
\end{equation}

On the other hand, if $q\leq \log^\frac{3}{2} x$ is exceptional, then, using Theorem 3 of \cite{LW}, the right-hand side of (\ref{eq39}) is at most
\begin{equation}
\frac{C_0\log\log x}{\log^{A+1} x}+\frac{x^{-\frac{4\pi}{9\times 0.4923\log^\frac{1}{4} x(\log\log x)^2}}}{1-\frac{4\pi}{9\times 0.4923\log^\frac{1}{4} x(\log\log x)^2}}+2x^\frac{1}{2}.
\end{equation}
We can easily confirm that, provided that $x\geq x_0$,
\begin{equation}
\frac{x^{-\frac{4\pi}{9\times 0.4923\log^\frac{1}{4} x(\log\log x)^2}}}{1-\frac{4\pi}{9\times 0.4923\log^\frac{1}{4} x(\log\log x)^2}}+2x^\frac{1}{2}<e^{-2000}\frac{x\log\log x}{\log^{A+1} x}.
\end{equation}

Hence (\ref{eq38}) gives
\begin{equation}
\begin{split}
\sum_{q\leq Q}\abs{\psi^{(R)}(x; q, a)-\psi^{(1)}(x; q, a)}
< & \sum_{q\leq Q}\frac{1}{\vph(q)}\sum_{\substack{1<r\leq R,\\ r\mid q}}\left(\vph^*(r)\omega(q)\log q+\frac{(C_0+e^{-100})x\log\log x}{\log^{A+1} x}\right)\\
= & \sum_{1<r\leq R}\vph^*(r)\sum_{\substack{q\leq Q.\\ r\mid q}}\frac{\omega(q)\log q}{\vph(q)}\\
& +\frac{(C_0+e^{-100})x\log\log x}{\log^{A+1} x}\sum_{1<r\leq R}\sum_{\substack{q\leq Q.\\ r\mid q}}\frac{1}{\vph(q)}.
\end{split}
\end{equation}
for $x\geq x_0$.

Since $\omega(n)<1.3841\log n/\log\log n$ by \cite[Theorem 11]{Rob}, the contribution of the first term is
\begin{equation}
\begin{split}
\leq & \frac{1.3841\log^2 Q}{\log\log Q}\sum_{1<r\leq R}\vph^*(r)\sum_{\substack{q\leq Q,\\ r\mid q}}\frac{1}{\vph(q)}\\
\leq & \frac{1.3841\log^2 Q}{\log\log Q}\sum_{1<r\leq R}\frac{\vph^*(r)}{\vph(r)}\sum_{\substack{l\leq \frac{Q}{r}}}\frac{1}{\vph(l)}\\
<&\frac{1.3841\log^2 Q}{\log\log Q}R C_2(1+\log Q)
<C_{10}\frac{x(\log\log x)^2}{\log^A x}.
\end{split}
\end{equation}
The contribution of the second term is
\begin{equation}\label{eq310}
\begin{split}
\leq & \frac{Cx\log\log x}{\log^{A+1} x}\sum_{1<r\leq R}\frac{1}{\vph(r)}\sum_{\substack{q\leq \frac{Q}{r}}}\frac{1}{\vph(q)}\\
<& \frac{(C_0+e^{-100})x\log\log x}{\log^{A+1} x}C_2^2\log R(1+\log Q) \\
<& \frac{(A+3)(C_0+e^{-100})C_2^2x(\log\log x)^2}{2\log^A x}.
\end{split}
\end{equation}
Thus we have
\begin{equation}
\sum_{q\leq Q}\abs{\psi^{(R)}(x; q, a)-\psi^{(1)}(x; q, a)}<\frac{(C_{10}+(A+3)(C_0+e^{-100})C_2^2)x(\log\log x)^2}{2\log^A x}.
\end{equation}
Combined with (\ref{eq37}), this yields
\begin{equation}\label{eq311}
\sum_{q\leq Q}\abs{\psi^{(1)}(x; q, a)}<\left(\frac{C_6+C_7+C_8}{\log\log x}+C_9+C_{10}+(A+3)(C_0+e^{-100})C_2^2\right)\frac{x(\log\log x)^2}{2\log^A x}.
\end{equation}

We take $\al_1$ from Table \ref{tbl2} and, for $A\in\{2, 3, 4, 5, 6, 7\}$ and $x\geq x_1$, we have
upper bounds given in Table \ref{tbl2} for $C_6, C_7$ and $C_9$.  Moreover, we see that $C_8, C_{10}<e^{-2000}$.

\begin{table}
\caption{upper bounds for constants in (\ref{eq311}) and Theorem \ref{thm1}}\label{tbl2}
\begin{center}
\begin{small}
\begin{tabular}{| c | c | c | c | c | c |}
 \hline
$A$ & $\al$ & $C_6$ & $C_7$ & $C_9$ & $C_1$ \\
 \hline
$2$ & $3.24$ & $4.3925995$ & $1.0176341$ & $5.4966132$ & $6.1079032$ \\
$3$ & $3.24$ & $4.3850253$ & $1.0100483$ & $6.5494453$ & $7.1364631$ \\
$4$ & $3.23$ & $4.3779004$ & $1.0055896$ & $7.6038431$ & $8.1716266$ \\
$5$ & $3.23$ & $4.3733770$ & $1.0009200$ & $8.6593226$ & $9.2113377$ \\
$6$ & $3.23$ & $4.3696917$ & $0.9971007$ & $9.7155886$ & $10.254346$ \\
$7$ & $3.23$ & $4.3666139$ & $0.9939087$ & $10.772447$ & $11.299829$ \\
 \hline
\end{tabular}
\end{small}
\end{center}
\end{table}

Now we can easily confirm that $\frac{C_6+C_7+C_8}{\log\log x}+C_9+C_{10}<C_1$ for $C_1$ given in Table \ref{tbl2}.
This completes the proof.

\section{}
\begin{lem}
Moreover, for $x\geq 7920$, we have
\begin{equation}\label{eq41}
\sum_{n\leq x}\frac{\mu^2(n)}{\vph(n)}\leq C_{11}+\log x
\end{equation}
where
\begin{equation}
C_{11}=1.334.
\end{equation}
\end{lem}
\begin{proof}
(\ref{eq22}) immedialtely gives
\begin{equation}
\begin{split}
\sum_{x_1<q\leq x}\frac{\mu^2(q)}{\vph(q)}=\log\frac{x}{x_1}+\theta\left[\frac{\zeta(3/2)}{\zeta(3)}\frac{\log\frac{x}{x_1}}{\sq{x}}+B_5\left(\frac{\sq{x}}{x_1}+\frac{1}{\sq{x}}\right)\right],
\end{split}
\end{equation}
where
\begin{equation}
B_5=\prod_p\left(1+\frac{2}{\sq{p}(p-1)}\right),
\end{equation}
and calculation gives that
\begin{equation}
\sum_{q\leq x}\frac{\mu^2(q)}{\vph(q)}\leq \log x+1.334
\end{equation}
for all $x\geq 7920$.
\end{proof}

Analogously to (\ref{eq35}) and (\ref{eq36}), we obtain
\begin{equation}
\sum_{q\leq Q}\mu^2(q)\abs{G^{(1)}(x; q, a)}< \frac{x\log\log x}{2e^\gamma\log^{A+1} x}+\left(1+\frac{\log x}{2Q}\right)\frac{x}{\log^{A+1} x}.
\end{equation}
and
\begin{equation}
\begin{split}
& \sum_{r<R}\sum_{q\leq Q, r\mid q}\frac{1}{\vph(q)}\sum_{\substack{b\pmod{q},\\ (b, q)=1}}\mu^2(q)\abs{G^{(1)}(x; q, b)} \\
&\quad <\sum_{r<R}\frac{x\log\log x}{2\vph(r)e^\gamma \log^{A+1} x}+\left(\frac{1}{r}+\frac{\log x}{2Q\vph(r)}\right)\frac{x}{\log^{A+1}x}\\
&\quad <\frac{(C_6+(A+3)\log\log x)x\log\log x}{2e^\gamma \log^A x}+\left(1+(A+3)\log\log x\right.\\
&\left.\qquad +\frac{(C_6+(A+3)\log\log x)\log x}{2Q}\right)\frac{x}{\log^{A+1}x}.
\end{split}
\end{equation}
Now (\ref{eq33}) gives
\begin{equation}
\sum_{q\leq Q}\abs{\psi^{(R)}(x; q, a)}\leq(C_6 xR^{-1}+(C_7+C_8)Q\sq{x})\log^3 x\log\log Q+C_{12}\frac{x(\log\log x)^2}{\log^A x}.
\end{equation}
Since $R=\log^{A+3} x$ and $Q=\sq{x}/\log^{A+3} x$, we have
\begin{equation}
\sum_{q\leq Q}\abs{\psi^{(R)}(x; q, a)}<\left(\frac{C_6+C_7+C_8}{\log\log x}+C_{12}\right)\frac{x(\log\log x)^2}{\log^A x}.
\end{equation}

Since we see that $R\geq (\log x_0)^{A+3}>7920$, we have
\begin{equation}
\begin{split}
\leq & \frac{x\log\log x}{\log^{A+1} x}\sum_{1<r\leq R}\frac{\mu^2(r)}{\vph(r)}\sum_{\substack{q\leq \frac{Q}{r}}}\frac{\mu^2(q)}{\vph(q)}\\
<& \frac{x\log\log x}{\log^{A+1} x}(C_{11}+\log R)(C_6+\log Q) \\
<& \frac{x(\log\log x)(C_{11}+(A+3)\log\log x)}{2\log^A x}\\
\leq & \frac{(A+3)C_{13}x(\log\log x)^2}{2\log^A x},
\end{split}
\end{equation}
Using this in place of (\ref{eq310}), we obtain
\begin{equation}
\sum_{q\leq Q}\mu^2(q)\abs{\psi^{(R)}(x; q, a)-\psi^{(1)}(x; q, a)}<\frac{(C_{10}+(A+3)(C_0+e^{-100})C_{13})x(\log\log x)^2}{2\log^A x}
\end{equation}
and therefore
\begin{equation}\label{eq312}
\sum_{q\leq Q}\mu^2(q)\abs{\psi^{(1)}(x; q, a)}<\left(\frac{C_6+C_7+C_8}{\log\log x}+C_{12}+C_{10}+(A+3)(C_0+e^{-100})C_{13}\right)\frac{x(\log\log x)^2}{2\log^A x}.
\end{equation}

\begin{table}
\caption{upper bounds for constants in (\ref{eq312}) and Theorem \ref{thm1}}\label{tbl3}
\begin{center}
\begin{small}
\begin{tabular}{| c | c | c |}
 \hline
$A$ & $C_{12}$ & $C_1^\prime$ \\
 \hline
$2$ & $1.5203771$ & $2.1316672$ \\
$3$ & $1.7963320$ & $2.3833498$ \\
$4$ & $2.0732743$ & $2.6410579$ \\
$5$ & $2.3508971$ & $2.9029123$ \\
$6$ & $2.6290135$ & $3.1677709$ \\
$7$ & $2.9075004$ & $3.4348830$ \\
 \hline
\end{tabular}
\end{small}
\end{center}
\end{table}

{}

{\small Tomohiro Yamada}\\
{\small Center for Japanese language and culture\\Osaka University\\562-8558\\8-1-1, Aomatanihigashi, Minoo, Osaka\\Japan}\\
{\small e-mail: \protect\normalfont\ttfamily{tyamada1093@gmail.com}}
\end{document}